\documentclass[12pt]{amsart}
\usepackage{amscd,verbatim}
\usepackage{amssymb}
\usepackage[all]{xy}

\newcommand{\A}{\mathbf{A}}

\newcommand{\G}{\mathbf{G}}
\renewcommand{\L}{\mathbb{L}}
\renewcommand{\P}{\mathbf{P}}

\newcommand{\Z}{\mathbf{Z}}
\newcommand{\sA}{\mathcal{A}}
\newcommand{\sB}{\mathcal{B}}
\newcommand{\sF}{\mathcal{F}}
\newcommand{\sG}{\mathcal{G}}
\newcommand{\sH}{\mathcal{H}}
\newcommand{\sO}{\mathcal{O}}
\newcommand{\bH}{\mathbb{H}}
\newcommand{\Mod}{\hbox{--}\operatorname{Mod}}
\newcommand{\Span}{\operatorname{\mathbf{Span}}}
\newcommand{\Cor}{\operatorname{\mathbf{Cor}}}
\newcommand{\Mack}{\operatorname{\mathbf{Mack}}}
\newcommand{\HI}{\operatorname{\mathbf{HI}}}
\newcommand{\PST}{\operatorname{\mathbf{PST}}}
\newcommand{\NST}{\operatorname{\mathbf{NST}}}
\newcommand{\DM}{\operatorname{\mathbf{DM}}}
\newcommand{\Hom}{\operatorname{Hom}}
\newcommand{\uHom}{\operatorname{\underline{Hom}}}

\newcommand{\IM}{\operatorname{Im}}
\newcommand{\Coker}{\operatorname{Coker}}
\newcommand{\Tr}{\operatorname{Tr}}

\newcommand{\Spec}{\operatorname{Spec}}
\newcommand{\oo}{\operatornamewithlimits{\otimes}\limits}
\newcommand{\by}[1]{\overset{#1}{\longrightarrow}}
\newcommand{\yb}[1]{\overset{#1}{\longleftarrow}}
\newcommand{\iso}{\by{\sim}}
\newcommand{\osi}{\yb{\sim}}
\newcommand{\eff}{{\operatorname{eff}}}
\newcommand{\Zar}{{\operatorname{Zar}}}
\newcommand{\Nis}{{\operatorname{Nis}}}
\newcommand{\surj}{\rightarrow\!\!\!\!\!\rightarrow}
\newcommand{\Surj}{\relbar\joinrel\surj}

\renewcommand{\phi}{\varphi}
\renewcommand{\epsilon}{\varepsilon}

\newcounter{spec}
{\end{list}}%

\swapnumbers
\newtheorem{lemma}{Lemma}[section]
\newtheorem{thm}[lemma]{Theorem}
\newtheorem{prop}[lemma]{Proposition}
\newtheorem{cor}[lemma]{Corollary}
\theoremstyle{definition}

\newtheorem{nota}[lemma]{Notation}
\newtheorem{para}[lemma]{}
\theoremstyle{remark}
\newtheorem{rk}[lemma]{Remark}

\numberwithin{equation}{section}

\begin{document}
\title[Somekawa's $K$-groups and Voevodsky's Hom groups]{Somekawa's $K$-groups and Voevodsky's
Hom groups (preliminary version)}
\author{Bruno Kahn}
\address{Institut de Math\'ematiques de Jussieu\\175--179, rue du Chevaleret\\75013 Paris\\France}
\email{kahn@math.jussieu.fr}
\date{September 22, 2010}
\begin{abstract}
We construct a surjective homomorphism from So\-me\-ka\-wa's $K$-group associated to a finite
collection of semi-abelian varieties over a perfect field to a corresponding Hom group in
Voevodsky's triangulated category of effective motivic complexes.
\end{abstract}
\maketitle

\tableofcontents

\section{Introduction}
In this note, we construct an epimorphism
\begin{equation}\label{eq1}
K(k;G_1,\dots,G_n)\Surj \Hom_{\DM_-^\eff}(\Z,G_1[0]\otimes\dots\otimes G_n[0])
\end{equation}
where $k$ is a perfect field, $G_1,\dots,G_n$ are semi-abelian $k$-varieties, the
left-hand-side is the abelian group defined by K. Kato and studied by M. Somekawa in
\cite{somekawa} and on the right hand side, the tensor product $G_1[0]\otimes\dots\otimes
G_n[0]$ is computed in Voevodsky's triangulated category of effective motivic complexes
\cite{voetri} or alternately in his category of homotopy invariant Nisnevich sheaves with
transfers (ibid.). This has been announced in \cite[Rk 10 (b)]{sp-ya} and is used in
\cite[Th. 3.9]{yamazaki}.

I expect \eqref{eq1} to be bijective. This would provide an affirmative answer to a
version of Somekawa's expectation in the introduction of his paper (probably the closest
answer, as long as one does not have an abelian category of mixed motives at hand). The
method I have in mind to prove injectivity involves defining a group
\[K_{-1}(k;G_1,\dots,G_n)\]
modelled on Voevodsky's construction $(-)_{-1}$ on sheaves. However the construction of this
group appears more subtle than I initially thought, so I felt it might be useful to
already release this much of the story.

Recall that $K(k;G_1,\dots,G_n)$ is defined as a quotient of a larger group
\[(G_1\oo^M\dots \oo^M G_n)(k)\]
where $\oo^M$ is computed in the category of cohomological Mackey functors \cite{knote}. Our
strategy will be as follows:

\begin{enumerate}
\item Construct a surjective homomorphism 
\begin{equation}\label{eq2}
(G_1\oo^M\dots \oo^M G_n)(k)\to \Hom_{\DM_-^\eff}(\Z,G_1[0]\otimes\dots\otimes G_n[0]).
\end{equation}
This is achieved in \S\S \ref{mor1} and \ref{mor2}.
\item Show that \eqref{eq2} factors through the Kato-Somekawa re\-la\-tions, yielding
\eqref{eq1}. This is achieved in Theorem \ref{p1}.
\end{enumerate}

I wish to thank Takao Yamazaki for his interest and encouragement to pursue this work.

\section{Mackey functors and presheaves with transfers}

\begin{para}A \emph{Mackey functor} over $k$ is a contravariant additive (i.e., commuting with
coproducts) functor $A$ from the category of \'etale $k$-schemes to the category of abelian
groups, provided with a covariant structure verifying the following exchange condition: if 
\[\begin{CD}
Y'@>f'>> Y\\
@Vg'VV @VgVV\\
X'@>f>> X
\end{CD}
\]
is a cartesian square of \'etale $k$-schemes, then the diagram
\[\begin{CD}
A(Y')@>{f'}^*>> A(Y)\\
@Vg'_*VV @Vg_*VV\\
A(X')@>f^*>> A(X)
\end{CD}
\]
commutes. Here, $^*$ denotes the contravariant structure while $_*$ denotes the covariant
structure. The Mackey functor $A$ is \emph{cohomological} if we further have
\[f_* f^* = \deg(f)\]
for any $f:X'\to X$, with $X$ connected. We denote by $\Mack$ the abelian category of Mackey
functors, and by $\Mack_c$ its full subcategory of cohomological Mackey functors.
\end{para}

\begin{para}Classically \cite[(1.4)]{thevenaz}, a Mackey functor may be viewed as a
contravariant additive functor on the category $\Span$ of ``spans" on \'etale $k$-schemes,
defined as follows: objects are \'etale $k$-schemes. A morphism from $X$ to $Y$ is an
equivalence class of diagram (span)
\begin{equation}\label{eq3}
X\yb{g} Z\by{f} Y.
\end{equation}

Composition of spans is defined via fibre product in an obvious manner. If $A$ is a Mackey
functor, the corresponding functor on $\Span$ has the same value on objects, while its value on
a span \eqref{eq3} is given by $g_*f^*$.

Note that $\Span$ is a preadditive category: one may add (but not substract) two morphisms
with same source and target. We may as well view a Mackey functor as an additive functor on the
associated additive category $\Z\Span$.
\end{para}

\begin{para}Let $\Cor$ be Voevodsky's category of finite correspondences on smooth
$k$-schemes, denoted by $SmCor(k)$ in \cite[\S 2.1]{voetri}. The category $\Z\Span$ is
isomorphic to its full subcategory consisting of smooth $k$-schemes of dimension $0$
(= \'etale $k$-schemes). In particular, any presheaf with transfers in the sense of Voevodsky
\cite[Def. 3.1.1]{voetri} restricts to a Mackey functor over $k$. By \cite[Cor. 3.15]{voepre},
the restriction of a \emph{homotopy invariant} presheaf with transfers yields a cohomological
Mackey functor. In other words, we have exact functors
\begin{align}
\rho:\PST&\to \Mack\label{eq4}\\
\rho:\HI&\to \Mack_c\label{eq5}
\end{align}
where $\PST$ denotes the category of presheaves with transfers (contravariant additive
functors from $\Cor$ to abelian groups) and $\HI$ is its full subcategory consisting of
homotopy invariant presheaves with transfers.
\end{para}

\begin{para} There is a tensor product of Mackey functors $\oo^M$, originally defined by L. G.
Lewis (unpublished): it extends naturally the symmetric
monoidal structure $(X,Y)\mapsto X\times_K Y$ on $\Z\Span$ via the additive Yoneda embedding
(see \S \ref{Atens}). If either $A$ or $B$ is cohomological, $A\oo^M B$ is cohomological. 

This tensor product is the same as the one defined in \cite{knote}: this follows from
\eqref{eqA.2} and the fact that $\Z\Span$ is rigid, all objects being self-dual
(indeed, $\Z\Span$ is canonically isomorphic to the category of Artin Chow motives with
integral coefficients).
\end{para}

\begin{para}
There is a tensor product on presheaves with transfers defined exactly in the same way \cite[p.
236]{voetri}.
\end{para}

\begin{para} By definition, the functor \eqref{eq4} equals $i^*$, where $i$ is the inclusion
$\Z\Span\to \Cor$. This inclusion has a left adjoint $\pi_0$ (scheme of constants). Both
functors
$i$ and
$\pi_0$ are symmetric monoidal: for $\pi_0$, reduce to the case where $k$ is separably closed.
\end{para}

\begin{para}\label{s2.5}  By \S \S \ref{Aadj} and \ref{Acoh}, this implies that \eqref{eq4} is
symmetric monoidal. In other words, if $\sF$ and $\sG$ are presheaves with transfers, then
\begin{equation}\label{eq6}
\rho\sF\oo^M\rho\sG\simeq \rho(\sF\otimes_{\PST}\sG).
\end{equation}
\end{para}

\begin{para}\label{s2.6} The inclusion functor $\HI\to \PST$ has a left adjoint $h_0$, and the
symmetric  monoidal structure of $\PST$ induces one on $\HI$ via $h_0$. In other words, if
$\sF,\sG\in \HI$, we define
\begin{equation}\label{eq8}
\sF\otimes_{\HI}\sG = h_0(\sF\otimes_{\PST} \sG).
\end{equation}

Note that \eqref{eq5} is \emph{not} symmetric monoidal (since it is the restriction of
\eqref{eq4}).
\end{para}

\begin{para}\label{ssurj} For any $\sF\in \PST$, the unit morphism $\sF\to h_0(\sF)$ induces a
surjection
\[\sF(k)\to h_0(\sF)(k).\]

This is obvious from the formula $h_0(\sF)=\Coker(C_1(\sF)\to \sF)$.
\end{para}

\begin{para}\label{s2.8} We shall also need to work with Nisnevich sheaves with transfers. We
denote by $\NST$ the category of Nisnevich sheaves with transfers (objects of $\PST$ which are
sheaves in the Nisnevich topology). By \cite[Th. 3.1.4]{voetri}, the inclusion functor $\NST\to
\PST$ has an exact left adjoint $\sF\mapsto \sF_\Nis$ (sheafification). The category $\NST$
then inherits a tensor product by the formula
\[\sF\otimes_{\NST} \sG=(\sF\otimes_{\PST}\sG)_\Nis.\]

Similarly, we define $\HI_\Nis=\HI\cap \NST$. The sheafification functor restricts to an exact
functor $\HI\to \HI_\Nis$ \cite[Th. 3.1.11]{voetri}, and $\HI_\Nis$ gets a tensor product by the
formula
\[\sF\otimes_{\HI_\Nis} \sG=(\sF\otimes_{\HI}\sG)_\Nis.\]

To summarise, all functors in the following commutative diagram are symmetric monoidal:
\begin{equation}\label{eq15}
\begin{CD}
\PST@>\Nis>> \NST\\
@V{h_0}VV @V{h_0^\Nis}VV\\
\HI@>\Nis>> \HI_\Nis.
\end{CD}
\end{equation}
where each functor is left adjoint to the corresponding inclusion.
\end{para}

\begin{para}\label{s3.3} Let $\sF$ be a presheaf on $Sm/k$, and let $\sF_\Nis$ be the
associated Nisnevich sheaf. Then we have an isomorphism
\begin{equation}\label{eq11}
\sF(k)\iso \sF_\Nis(k).
\end{equation}

Indeed, any covering of the Nisnevich topology on $\Spec k$ refines to a trivial covering. In
particular, the functor $\sF\mapsto \sF_\Nis(k)$ is exact.

This applies in particular to a presheaf with transfers and the associated Nisnevich sheaf with
transfers.
\end{para}

\begin{para}\label{mor1} If $G$ is an abelian $k$-group scheme whose identity component is a
quasi-projective variety, then $G$ has a canonical structure of Nisnevich sheaf with transfers
(\cite[proof of Lemma 3.2]{spsz} completed by \cite[Lemma 1.3.2]{bar-kahn}). This applies in
particular to semi-abelian varieties. In particular, if $G_1,\dots,G_n$ are semi-abelian
varieties, \eqref{eq6} yields a canonical isomorphism
\begin{equation}\label{eq7}
(G_1\oo^M\dots \oo^MG_n)(k)\simeq (G_1\otimes_{\PST}\dots \otimes_{\PST}G_n)(k)
\end{equation}
where the $G_i$ are considered on the left as Mackey functors, and on the right as presheaves
with transfers. 

Since the $G_i$ are semi-abelian varieties, they are homotopy invariant. Therefore, composing
\eqref{eq7} with the unit morphism $Id\Rightarrow h_0^\Nis$ from \eqref{eq15} and taking
\eqref{eq8} into account, we get a canonical morphism
\begin{equation}\label{eq9}
(G_1\oo^M\dots \oo^MG_n)(k)\to (G_1\otimes_{\HI_\Nis}\dots \otimes_{\HI_\Nis}G_n)(k).
\end{equation}
which is surjective by \S \ref{ssurj}.
\end{para}

\section{Preseheaves with transfers and motives}

\begin{para} The left adjoint $h_0^\Nis$ in \eqref{eq15} ``extends" to a left adjoint $C_*$ of
the inclusion
\[\DM_-^\eff\to D^-(\NST)\]
where the left hand side is Voevodsky's triangulated category of effective motivic complexes
\cite[\S 3, esp. Prop. 3.2.3]{voetri}. 

More precisely, $\DM_-^\eff$ is defined as the full subcategory of objects of $D^-(\NST)$
whose cohomology sheaves are homotopy invariant. The canonical $t$-structure of $D^-(\NST)$
induces a $t$-structure on $\DM_-^\eff$, with heart $\HI_\Nis$. The functor $C_*$ is right
exact with respect to these $t$-structures, and if $\sF\in \NST$, then
$H_0(C_*(\sF))=h_0^\Nis(\sF)$.
\end{para}

\begin{para} The tensor structure of \S \ref{s2.8} on $\NST$ extends to one on $D^-(\NST)$
\cite[p. 206]{voetri}. Via
$C_*$, this tensor structure descends to a tensor structure on $\DM_-^\eff$ \cite[p.
210]{voetri}, which will simply be denoted by
$\otimes$. The relationship between this tensor structure and the one of \S \ref{s2.8} is as
follows: is $\sF,\sG\in \HI_\Nis$, then
\begin{equation}\label{eq10}
\sF\otimes_{\HI_\Nis}\sG = H^0(\sF[0]\otimes \sG[0])
\end{equation}
where $\sF[0],\sG[0]$ are viewed as complexes of Nisnevich sheaves with transfers concentrated
in degree $0$.

We shall need the following lemma, which is not explicit in \cite{voetri}:
\end{para}

\begin{lemma}\label{l2} The tensor product $\otimes$ of $\DM_-^\eff$ is right exact with respect to the homotopy $t$-structure.
\end{lemma}

\begin{proof} By definition,
\[C\otimes D = C_*(C\oo^L D)\]
for $C,D\in \DM_-^\eff$, where $\oo^L$ is the tensor product of $D^-(\NST)$ defined in
\cite[p. 206]{voetri}. We want to show that, if $C$ and $D$ are concentrated in degrees $\le
0$, then so is $C\otimes D$. Using the canonical left resolutions of loc. cit., it is enough to
do it for $C$ and $D$ of the form $C_*(L(X))$ and $C_*(L(Y))$ for two smooth schemes $X,Y$.
Since $C_*$ is symmetric monoidal, we have
\[C_*(L(X))\otimes C_*(L(Y))\osi C_*(L(X)\oo^L L(Y)) = C_*(L(X\times Y))\]
and the claim is obvious in view of the formula for $C_*$ \cite[p. 207]{voetri}.
\end{proof}

\begin{para} Let $C\in \DM_-^\eff$. For any $X\in Sm/k$ and any $i\in\Z$, we have
\[\bH^i_\Nis(X,C)\simeq \Hom_{\DM_-^\eff}(M(X),C[i])\]
where $M(X)=C_*(L(X))$ is the motive of $X$ computed in $\DM_-^\eff$ (cf. \cite[Prop.
3.2.7]{voetri}).

Specialising to the case $X=\Spec k$ ($M(X)=\Z$) and taking \S \ref{s3.3} into account, we get
\begin{equation}\label{eq12}
\Hom_{\DM_-^\eff}(\Z,C[i])\simeq H^i(C)(k).
\end{equation}

Combining \eqref{eq10}, \eqref{eq11} and \eqref{eq12}, we get:
\end{para}

\begin{lemma} \label{l1} Let $\sF_1,\dots,\sF_n$ be homotopy invariant Nisnevich sheaves with
transfers. Then we have a canonical isomorphism
\[(\sF_1\otimes_{\HI_\Nis}\dots \otimes_{\HI_\Nis}\sF_n)(k) \simeq
\Hom_{\DM_-^\eff}(\Z,\sF_1[0]\otimes\dots \otimes\sF_n[0]).\]
\qed
\end{lemma}

\begin{para}\label{mor2} Combining Lemma \ref{l1} with  \eqref{eq9}, we get the announced
homomorphism
\eqref{eq2}. However, we shall mainly work with presheaves with transfers in the sequel, hence
use \eqref{eq9} rather than \eqref{eq2}.
\end{para}

\section{Presheaves with transfers and local symbols}

\begin{para}\label{s4.2} Given a presheaf with transfers $\sG$, recall from \cite[p.
96]{voepre} the presheaf with transfers $\sG_{-1}$ defined by the formula
\[\sG_{-1}(U) = \Coker\left(\sG(U\times \A^1)\to \sG(U\times (\A^1-\{0\}))\right).\]

Suppose that $\sG$ is homotopy invariant. Let $X\in Sm/k$ (connected),  $K=k(X)$ and $x\in X$ be
a point of codimension $1$. By \cite[Lemma 4.36]{voepre}, there is a canonical isomorphism
\begin{equation}\label{eq14a}
\sG_{-1}(k(x))\simeq H^1_x(X,\sG_\Zar)
\end{equation}
yielding a
canonical map
\begin{equation}\label{eq14}
\partial_x:\sG(K)\to \sG_{-1}(k(x)).
\end{equation}
\end{para}

The following lemma follows from the construction of the isomorphisms \eqref{eq14a}. It is part
of the general fact that $\sG$ defines a cycle module in the sense of Rost (cf. \cite[Prop.
5.4.64]{deglise}).

\begin{lemma}\label{l4} a) Let $f:Y\to X$ be a dominant morphism, with
$Y$ smooth and connected. Let $L=k(Y)$, and let $y\in Y^{(1)}$ be such that $f(y)=x$. Then the
diagram
\[\begin{CD}
\sG(L)@>(\partial_y)>> \sG_{-1}(k(y))\\
@Af^*AA @Ae f^*AA\\
\sG(K)@>\partial_x>> \sG_{-1}(k(x))
\end{CD}\]
commutes, where $e$ is the ramification index of $v_y$ relative to $v_x$.\\
b) If $f$ is finite surjective, the diagram
\[\begin{CD}
\sG(L)@>(\partial_y)>> \displaystyle\bigoplus_{y\in f^{-1}(x)} \sG_{-1}(k(y))\\
@Vf_*VV @Vf_*VV\\
\sG(K)@>\partial_x>> \sG_{-1}(k(x))
\end{CD}\]
commutes.\qed
\end{lemma}

\begin{prop}\label{p2} Let $\sG\in \HI_\Nis$. There is a canonical isomorphism
\[\sG_{-1} = \uHom(\G_m,\sG).\]
\end{prop}

\begin{proof} This may not be the most economic proof, but it is quite short. The statement
means that $\sG_{-1}$ represents the functor
\[\sH\mapsto \Hom_{\HI_\Nis}(\sH\otimes_{\HI_\Nis}\G_m,\sG).\]

By \cite[Lemma 4.35]{voepre}, we have
\[\sG_{-1} = \Coker(\sG\to p_*p^*\sG)\]
where $p:\A^1-\{0\}\to \Spec k$ is the structural morphism and $p_*, p^*$ are computed with
respect to the Zariski topology. By \cite[Th. 5.7]{voepre}, we may replace the Zariski topology
by the Nisnevich topology. Moreover, by \cite[Prop. 5.4 and Prop. 4.20]{voepre}, we have
$R^ip_*p^*\sG = 0$ for $i>0$, hence $p_*p^*\sG[0]\iso Rp_*p^*\sG[0]$.

By \cite[Prop. 3.2.8]{voetri}, we have
\[Rp_*p^* \sG[0] = \uHom(M(\A^1-\{0\}),\sG[0])\]
where $\uHom$ is the (partially defined) internal Hom of $\DM_-^\eff$. By \cite[Prop.
3.5.4]{voetri} (Gysin triangle) and homotopy invariance, we have an exact triangle, split by
any rational point of $\A^1-\{0\}$:
\[\Z(1)[1]\to M(\A^1-\{0\})\to \Z\by{+1}\]

To get a canonical splitting, we may choose the rational point $1\in \A^1-\{0\}$.

By \cite[Cor. 3.4.3]{voetri}, we have an isomorphism $\Z(1)[1]\simeq \G_m[0]$. Hence, in
$\DM_-^\eff$, we have an isomorphism
\[\sG_{-1}[0]\simeq \uHom(\G_m[0],\sG[0]).\]

Let $\sH\in \HI_\Nis$. We get:
\begin{multline*}
\Hom_{\DM_-^\eff}(\sH[0],\sG_{-1}[0]) \simeq \Hom_{\DM_-^\eff}(\sH[0]\otimes \G_m[0],\sG[0])\\
\simeq \Hom_{\HI_\Nis}(H^0(\sH[0]\otimes \G_m[0]),\sG)=:\Hom_{\HI_\Nis}(\sH\otimes_{\HI_\Nis}
\G_m,\sG)
\end{multline*}
as desired (see \eqref{eq10}). For the second isomorphism, we have used the right exactness of
$\otimes$ (Lemma
\ref{l2}).
\end{proof}

\begin{rk}\label{r1} The proof of Proposition \ref{p2} also shows that, in $\DM_-^\eff$, we
have an isomorphism
\[\uHom(\G_m[0],\sG[0])\simeq \uHom(\G_m,\sG)[0]\]
where the left $\uHom$ is computed in $\DM_-^\eff$ and the right $\uHom$  is computed in
$\HI_\Nis$. In particular, $\uHom(\G_m[0],-):\DM_-^\eff\to \DM_-^\eff$ is $t$-exact.
\end{rk}

\begin{prop}\label{p4} Let $C$ be a smooth, proper, connected curve over $k$,
with function field $K$. There exists a canonical homomorphism
\[\Tr_{C/k}:H^1_\Zar(C,\sG)\to \sG_{-1}(k)\]
such that, for any $x\in C$, the composition
\[\sG_{-1}(k(x))\simeq H^1_x(C,\sG)\to H^1_\Zar(C,\sG)\by{\Tr_C} \sG_{-1}(k)\]
equals the transfer map $\Tr_{k(x)/k}$ associated to the finite surjective morphism $\Spec
k(x)\to\Spec k$.
\end{prop}

\begin{proof} By  
\cite[Prop. 3.2.7]{voetri}, we have
\[H^1_\Zar(C,\sG)\iso H^1_\Nis(C,\sG)\simeq
\Hom_{\DM_-^\eff}(M(C),\sG[1]).\] 

The structural morphism $C\to\Spec k$ yields a morphism of motives $M(C)\to \Z$ which, by
Poincar\'e duality, yields a canonical morphism
\[\G_m[1]\simeq \Z(1)[2]\to M(C).\]

(One may view this morphism as the image of the canonical morphism $\L\to h(C)$ in the category
of Chow motives.)

Therefore, by Proposition \ref{p2} and Remark \ref{r1}, we get a map
\[\Tr_{C/k}:H^1_\Zar(X,\sG)\to\Hom_{\DM_-^\eff}(\G_m[1],\sG[1])=\sG_{-1}(k).\]

It remains to prove the claimed compatibility. Let $M^x(C)$ be the motive of $C$ with supports
in $x$, defined as $C_*(\Coker(L(C-\{x\})\to L(C))$. Let $\Z_{k(x)}=M(\Spec k(x))$. By
\cite[proof of Prop. 3.5.4]{voetri}, we have an isomorphism $M^x(C)\simeq \Z_{k(x)}(1)[2]$, and
we have to show that the composition
\[
\Z(1)[2]\to M(C)\by{g_x} \Z_{k(x)}(1)[2]
\]
is $\Tr_{k(x)/k}$, up to twisting and shifting.  To see this, we observe that $g_x$ is the image of the morphism of Chow motives
\[h(C)\to h(\Spec k(x))(1)\]
dual to the morphism $h(\Spec k(x))\to h(C)$ induced by the inclusion $\Spec k(x)\to C$: this
is easy to check from the definition of $g_x$ in \cite{voetri} (observe that in this special
case, $Bl_x(C)=C$ and that we may use a variant of the said construction replacing $C\times
\A^1$ by $C\times\P^1$ to stay within smooth projective varieties). The conclusion now follows
from the fact that the composition 
\[\Spec k(x)\to C\to \Spec k\]
is the structural morphism of $\Spec k(x)$.
\end{proof}

\begin{prop}[Reciprocity] \label{p3} Let $C$ be a smooth, proper, connected curve over $k$, with
function field $K$. Then the sequence
\[\begin{CD}
\sG(K)@>{(\partial_x)}>> \bigoplus_{x\in C} \sG_{-1}(k(x))@>\sum_x \Tr_{k(x)/k}>>
\sG_{-1}(k)
\end{CD}\]
is a complex.
\end{prop}

\begin{proof} This follows from Proposition \ref{p4}, since the composition
\[\sG(K)\to \bigoplus_{x\in C} H^1_x(C,\sG)\by{g_x} H^1(C,\sG)\]
is $0$.
\end{proof}

\begin{para}\label{s4.4}
If $\sF,\sG$ are presheaves with transfers, there is a bilinear morphism of presheaves with
transfers (i.e. a natural transformation over $\PST\times \PST$):
\begin{multline*}
\sF(U)\otimes \sG_{-1}(V) =\\
\Coker\left(\sF(U)\otimes \sG(V\times\A^1)\to \sF(U)\otimes
\sG(V\times(\A^1-\{0\}))\right)\to\\
 \Coker\left((\sF\otimes_{\PST} \sG)(U\times V\times \A^1)\to (\sF\otimes_{\PST}
\sG)(U\times V\times(\A^1-\{0\}))\right)\\
= (\sF\otimes_{\PST}\sG)_{-1}(U\times V)
\end{multline*}
which induces a morphism
\begin{equation}\label{eq16}
\sF\otimes_{\PST}\sG_{-1}\to (\sF\otimes_{\PST}\sG)_{-1}.
\end{equation}

In particular, for $\sG=\G_m$, we get a morphism $\sF\to (\sF\otimes_{\PST} \G_m)_{-1}$.
\end{para}

\begin{thm}\label{t1} Suppose $\sF\in \HI_\Nis$. Then\\
a) The composition 
\[\sF\to (\sF\otimes_{\PST} \G_m)_{-1}\to (\sF\otimes_{\HI_\Nis} \G_m)_{-1}\]
is the unit map of the adjunction between $-\otimes_{\HI_\Nis} \G_m$ and $(-)_{-1}$
stemming from Proposition \ref{p2}.\\
b) This composition is
an isomorphism.
\end{thm}

\begin{proof} a) is an easy bookkeeping. For b),  we compute again in $\DM_-^\eff$. By
Proposition
\ref{p2}, we are considering the  morphism in $\HI_\Nis$
\begin{equation}\label{eq13}
\sF\to \uHom(\G_m,\sF\otimes_{\HI_\Nis} \G_m).
\end{equation}

Consider the corresponding morphism in $\DM_-^\eff$
\[\sF[0]\to \uHom(\G_m[0],\sF[0]\otimes \G_m[0]).\]

As recalled in the proof of Proposition \ref{p2}, we have $\G_m[0]=\Z(1)[1]$, hence the above
morphism amounts to
\[\sF[0]\to \uHom(\Z(1),\sF[0](1))\]
which is an isomorphism by the cancellation theorem \cite{voecan}. A fortiori, \eqref{eq13},
which is (by Remark \ref{r1}) the $H^0$ of this isomorphism, is an isomorphism.
 \end{proof}

\begin{nota}\label{n1} Let $\sF,\sG\in \HI_\Nis$ and
$\sH=\sF\otimes_{\HI_\Nis}
\sG$. Let
$X,K,x$ be as in \S \ref{s4.2}. For $(a,b)\in \sF(K)\times \sG(K)$,  we denote by
$a\cdot b$ the image of $a\otimes b$ in $\sH(K)$ by the map
\[\sF(K)\otimes \sG(K)\to \sH(K).\]
\end{nota}

\begin{prop}[cf. \protect{\cite[Prop. 5.5.27]{deglise}}]\label{l3} Let $\sF,\sG\in \HI_\Nis$,
and consider the morphism induced by \eqref{eq16}
\[\sF\otimes_{\HI_\Nis} \sG_{-1}\by{\Phi} (\sF\otimes_{\HI_\Nis} \sG)_{-1}.\]
 Let
$X,K,x$ be as in \S \ref{s4.2}. Then the diagram
\[\xymatrix{
\sF(\sO_{X,x})\otimes \sG(K)\ar[r]\ar[d]^{i_x^*\otimes \partial_x}&
(\sF\otimes_{\HI_\Nis}\sG)(K)\ar[dd]^{\partial_x}\\
\sF(k(x))\otimes \sG_{-1}(k(x))\ar[d]\\
(\sF\otimes_{\HI_\Nis} \sG_{-1})(k(x))\ar[r]^{\Phi}& (\sF\otimes_{\HI_\Nis} \sG)_{-1}(k(x))
}\]
commutes, where $i_x^*$ is induced by the reduction map $\sO_{X,x}\to k(x)$. In other words,
with Notation \ref{n1} we have the identity
\begin{equation}\label{eq17}
\partial_x(a\cdot b) = \Phi(i_x^*a\cdot \partial_x b)
\end{equation}
for $(a,b)\in \sF(\sO_{X,x})\times \sG(K)$.
\end{prop}

\begin{cor}\label{c1} Let $\sF\in \HI_\Nis$; let $X,K,x$ be as in \S \ref{s4.2} and let
$(a,f)\in \sF(K)\times K^*$. Let $\sH=\sF\otimes_{\HI_\Nis}\G_m$, and consider the element
$a\cdot f\in\sH(K)$ as in Notation \ref{n1}.\\ 
a) Suppose that $a\in \IM(\sF(\sO_{X,x})\to \sF(K))$. Then
we have
\[\partial_x(a\cdot f) = v_x(f) \bar a\]
where $\partial_x$ is the map of \eqref{eq14} and $\bar a$ is the image of $a$ in $\sF(k(x))$.
Here we have used the isomorphism $\sH_{-1}\simeq \sF$ of Theorem \ref{t1}.\\
b) Suppose that $v_x(f- 1)>0$. Then $\partial_x(a\cdot f) = 0$.
\end{cor}

\begin{proof} a) This follows from Proposition \ref{l3} (applied with $\sG=\G_m$) and Theorem
\ref{t1}. b) This follows again from Proposition \ref{l3}, after switching the r\^oles of $\sF$
and $\sG$.
\end{proof}

\begin{prop}\label{l5} Let $G$ be a semi-abelian variety and let $\sH=G\otimes_{\HI_\Nis}
\G_m$, so that $\sH_{-1} = G$ by Theorem \ref{t1}. Let $X,K,x$ be as in \S \ref{s4.2}, and let
$(g,f)\in G(K)\times K^*$. As in Corollary \ref{c1}, write $g\cdot f$ for the image of
$g\otimes f$ in $\sH(K)$. Then $\partial_x(g\cdot f)=\partial_x(g,f)$, where
$\partial_x(-,-)$ is Somekawa's  local symbol \cite[(1.1)]{somekawa} (generalising the
Rosenlicht-Serre local symbol).
\end{prop}

\begin{proof} Up to base-changing from $k$ to $\bar k$ (see Lemma \ref{l4} a)), we may assume
$k$ algebraically closed. We then have to show that $\partial_x(g\cdot f)$ is the
Rosenlicht-Serre local symbol of
\cite[Ch. III, Def. 2]{gacl}. In this definition, Condition i) is obvious, Condition ii) is
Corollary \ref{c1} b), Condition iii) is Corollary \ref{c1} a) and Condition iv) is Proposition
\ref{p3}. The conclusion now follows from the uniqueness of the local symbol \cite[Ch. III,
Prop. 1]{gacl}.
\end{proof}

\section{The factorisation}

\begin{thm}\label{p1} The homomorphism \eqref{eq9} factors through So\-me\-ka\-wa's relations.
Consequently, we get a surjective homomorphism \eqref{eq1}.
\end{thm}

\begin{proof} Let $C/k$ be a smooth proper an connected curve. Let $K=k(C)$, and, for all
$i\in [1,n]$, let $g_i\in G_i(K)$. We also give ourselves a rational function $h\in K^*$. We
assume that, for any $c\in C$, there exists $i(c)$ such that $g_i\in G_i(\sO_{C,c})$ for all
$i\ne i(c)$. Let $\sF=G_1\otimes_{\HI_\Nis}\dots
\otimes_{\HI_\Nis}G_n$. We must show that the element
\[\sum_{c\in C} \Tr_{k(c)/k}(g_1(c)\otimes \dots \otimes
\partial_c(g_{i(c)},h)\otimes\dots\otimes g_n(c)) \] of $(G_1\oo^M\dots \oo^M G_n)(k)$ goes to
$0$ in $\sF(k)$ via \eqref{eq9}, where $\partial_c$ is Somekawa's local symbol.\footnote{As
was observed by W. Raskind, the signs appearing in
\protect\cite[(1.2.2)]{somekawa} should not be there.}

Consider the element $g=g_1\otimes\dots\otimes g_n\in \sF(K)$. It follows from Corollary
\ref{c1} a) and Proposition  \ref{l5} that, for any $c\in C$, we have
\begin{multline*}
g_1(c)\otimes \dots \otimes \partial_c(g_{i(c)},h)\otimes\dots\otimes g_n(c)=\\
g_1(c)\otimes \dots \otimes \partial_c(g_{i(c)}\otimes\{h\})\otimes\dots\otimes g_n(c)=
\partial_c(g\otimes \{h\}).
\end{multline*}

The claim now follows from Proposition \ref{p3}.
\end{proof}

\appendix

\section{Extending monoidal structures}\label{stens}

\begin{para} Let $\sA$ be an additive category. We write $\sA\Mod$ for the category of
contravariant additive functors from $\sA$ to abelian groups. This is a Grothendieck abelian
category. We have the additive Yoneda embedding
\[y_\sA:\sA\to \sA\Mod\]
sending an object to the corresponding representable functor.
\end{para}

\begin{para}\label{Aadj} Let $f:\sA\to \sB$ be an additive functor. We have an induced functor
$f^*:\sB\Mod\to \sA\Mod$ (``composition with $f$"). As in \cite[Exp. 1, Prop. 5.1 and
5.4]{sga4}, the functor
$f^*$ has a left adjoint
$f_!$ and a right adjoint $f_*$ and the diagram
\[\begin{CD}
\sA@>y_\sA>> \sA\Mod\\
@VfVV @Vf_!VV\\
\sB@>y_\sB>> \sB\Mod
\end{CD}\]
is naturally commutative.
\end{para}

\begin{para} If $f$ is fully faithful, then $f_!$ and $f_*$ are fully faithful and $f^*$ is a
localisation, as in \cite[Exp. 1, Prop. 5.6]{sga4}. 
\end{para}

\begin{para} Suppose that $f$ has a left adjoint $g$. Then we have natural isomorphisms
\[g^*\simeq f_!,\quad g_*\simeq f^*\]
as in \cite[Exp. 1, Prop. 5.5]{sga4}.
\end{para}

\begin{para} Suppose further that $f$ is fully faithful. Then $g^*\simeq f_!$ is fully faithful.
From the composition 
\[g^*g_*\Rightarrow Id_{\sA\Mod}\Rightarrow g^*g_!\]
of the unit with the counit, one then deduces a canonical morphism of functors
\[g_*\Rightarrow g_!.\]
\end{para}

\begin{para}\label{Atens} Let $\sA$ and $\sB$ be two additive categories. Their \emph{tensor
product} is the category $\sA\boxtimes \sB$ whose objects are finite collections $(A_i,B_i)$
with
$(A_i,B_i)\in \sA\times\sB$, and
\[(\sA\boxtimes\sB)((A_i,B_i),(C_j,D_j)) = \bigoplus_{i,j}\sA(A_i,C_j)\otimes \sB(B_i,D_j).\]

We have a ``cross-product" functor
\[\boxtimes:\sA\Mod\times \sB\Mod\to (\sA\boxtimes\sB)\Mod\]
given by
\[(M\boxtimes N)((A_i,B_i)) = \bigoplus_i M(A_i)\otimes N(B_i).\]
\end{para}

\begin{para}\label{A!} Let $\sA$ be provided with a biadditive bifunctor $\bullet:\sA\times
\sA\to
\sA$. We may view $\bullet$ as an additive functor $\sA\boxtimes \sA\to \sA$. We may then extend
$\bullet$ to $\sA\Mod$ by the composition
\[\sA\Mod\times \sA\Mod\by{\boxtimes}(\sA\boxtimes \sA)\Mod\by{\bullet_!}\sA\Mod.\]

This is an extension in the sense that the diagram
\[\begin{CD}
\sA\times \sA@>y_\sA\times y_\sA>> \sA\Mod\times \sA\Mod\\
@V\bullet\times\bullet VV @V\bullet VV\\
\sA@>y_\sA>> \sA\Mod 
\end{CD}\]
is naturally commutative.

If $\bullet$ is monoidal (resp. monoidal symmetric), then its associativity and commutativity
constraints canonically extend to $\sA\Mod$.
\end{para}

\begin{para}\label{Acoh} Let $\sA,\sB$ be two  additive symmetric monoidal categories, and let
$f:\sA\to
\sB$ be an additive symmetric monoidal functor. The above definition shows that the functor
$f_!:\sA\Mod\to \sB\Mod$ is also symmetric monoidal.
\end{para}

\begin{para} In \S\ref{A!}, let us write $\bullet_!= \int$ for clarity. Let $P\in (\sA\boxtimes
\sA)\Mod$. Then $\int P$ is the \emph{left Kan extension of $P$ along $\bullet$} in the sense
of \cite[X.3]{mcl}. This gives a formula for $\int P$ as a \emph{coend} (ibid., Th. X.4.1); for
$A\in\sA$:
\begin{equation}\label{eqA.1}
\int P(A) = \int^{(B,B')} \sA(A,B\bullet B')\otimes P(B,B').
\end{equation}

In particular:
\end{para}

\begin{prop} Suppose $\sA$ rigid. Then \eqref{eqA.1} simplifies as
\[\int P(A) = \int^B  P(B, A\bullet B^*)\]
where $B^*$ is the dual of $B\in \sA$. In particular, if $P=M\boxtimes N$ for $M,N\in \sA\Mod$,
we have for $A\in\sA$:
\begin{equation}\label{eqA.2}
(M\bullet N)(A)= \int^B  M(B)\otimes N(A\bullet B^*)
\end{equation}
which describes $M\bullet N$ as a ``convolution".
\end{prop}

\begin{proof} Applying \eqref{eqA.1} and rigidity, we have
\begin{multline*}
\int P(A) = \int^{(B,B')} \sA(A,B\bullet B')\otimes P(B,B')\\
=\int^{(B,B')} \sA(A\bullet B^*, B')\otimes P(B,B')\\
=\int^B  P(B,A\bullet B^*)
\end{multline*}
because in the third formula, the variable $B'$ is dummy (this simplification is not in Mac
Lane!).
\end{proof}

\enlargethispage*{20pt}

\end{document}